\documentclass[preprint,12pt,3p]{elsarticle}
\usepackage{verbatim}
\usepackage{amsmath,amsthm,amsfonts,amssymb}
\usepackage{algorithm}
\usepackage{algpseudocode}
\usepackage{cases}
\usepackage{hyperref}

\usepackage[normalem]{ulem}
\usepackage{xcolor,sectsty}
\definecolor{astral}{RGB}{46,116,181}
\subsectionfont{\color{astral}}
\sectionfont{\color{astral}}
\newtheorem{theorem}{Theorem}[section]
\newtheorem{lemma}[theorem]{Lemma}

\newtheorem{definition}[theorem]{Definition}
\newtheorem{example}[theorem]{Example}

\newtheorem{remark}[theorem]{Remark}
\newtheorem{thm}{Theorem}[section]

\newtheorem{corollary}[thm]{Corollary}

\definecolor{darkslategray}{rgb}{0.18, 0.31, 0.31}
\definecolor{warmblack}{rgb}{0.0, 0.26, 0.26}

\journal{arXiv.org}

\newcommand{\R}{{\mathbb R}}
\newcommand{\C}{{\mathbb C}}

\newcommand{\mc}[1]{\mathcal {#1}}
\newcommand{\dg}{{\dagger}}
\newcommand{\n}{{*_N}}
\newcommand{\m}{{*_M}}
\newcommand{\kp}{{*_K}}
\newcommand{\lp}{{*_L}}
\newcommand{\1}{{*_1}}
\newcommand{\2}{{*_2}}

\begin{document}

\begin{frontmatter}

\title{ \textcolor{warmblack}{\bf An extension  of the Moore-Penrose Inverse of a Tensor via the Einstein product}}

\author{Krushnachandra Panigrahy$^\dag$$^a$, Debasisha Mishra$^\dag$$^b$}

\address{               $^{\dag}$Department of Mathematics,\\
                        National Institute of Technology Raipur,\\
                        Raipur, Chhattisgarh, India.\\
                        \textit{E-mail$^a$}: \texttt{kcp.224\symbol{'100}gmail.com }\\
                        \textit{E-mail$^b$}: \texttt{dmishra\symbol{'100}nitrr.ac.in. }
        }

\begin{abstract}
\textcolor{warmblack}{ In this paper, we first give an expression
for the Moore-Penrose inverse of the  product of two tensors via the
Einstein product. We then introduce a new generalized inverse of a
tensor called \textit{product Moore-Penrose inverse}. A necessary
and sufficient condition for the coincidence of the Moore-Penrose
inverse and the product Moore-Penrose inverse is also proposed.
Finally,  the \textit{triple reverse order law} of tensors  is
introduced.}
\end{abstract}


\end{frontmatter}

\section{Introduction}\label{sec1}
Research on tensors has been very active recently \cite{bu1, bu2,
ding, klda, li, sh1, sh2, yw, yng} as tensors have many applications
in different fields like graph analysis, computer vision, signal
processing, data mining, and chemo-metrics, etc. (see  \cite{burmc,
cmn, cpr, eldn, hu, smilde, vla} and the references cited there in).
Let $\mathbb{C}^{I_{1}\times\cdots\times I_{N}}$  be the set of
order $N$ and dimension $I_1 \times \cdots \times I_N$ tensors over
the complex field $\mathbb{C}$.  $\mc{A} \in
\mathbb{C}^{I_{1}\times\cdots\times I_{N}}$ is a multiway array with
$N$-th order tensor, and $I_{1}, I_{2}, \cdots, I_{N}$ are
dimensions of the first, second,$\cdots$, $N$th way, respectively.
Each entry of $\mc{A}$ is denoted by $a_{i_{1}...i_{N}}$. Throughout
the paper, tensors are represented in calligraphic letters like
$\mc{A}$, and the notation $(\mc{A})_{i_{1}...i_N}=
a_{i_{1}...i_{N}}$ represents the scalars. Let $\mc{A}\in
{\C}^{I_{1}\times\cdots\times I_{M}\times J_{1}\times \cdots \times
J_{N}}$, then its {\it conjugate transpose}, denoted by
$\mc{A}^{H}$, is defined as $(\mc{A}^{H})_{j_{1}\hdots
j_{N}i_{1}\hdots i_{M}}=\overline{\mc{A}}_{i_{1}\hdots
i_{M}j_{1}\hdots j_{N}},$ where the over-line stands for the
conjugate of $\mc{A}_{i_{1}\hdots i_{M}j_{1}\hdots j_{N}}$. If the
tensor $\mc{A}$ is real, then its {\it transpose} is denoted by
$\mc{A}^{T}$, and is defined as $(\mc{A}^{T})_{j_{1}\hdots
j_{N}i_{1}\hdots i_{M}}=\mc{A}_{i_{1}\hdots i_{M}j_{1}\hdots
j_{N}}$. The Einstein product \cite{ein} $ \mc{A}\n\mc{B} \in
\mathbb{C}^{I_1\times\cdots\times I_N \times J_1 \times\cdots\times
J_M }$ of tensors $\mc{A} \in \mathbb{C}^{I_{1}\times\cdots\times
I_{N} \times K_{1} \times\cdots\times K_{N} }$ and $\mc{B} \in
\mathbb{C}^{K_{1}\times\cdots\times K_{N} \times J_{1}
\times\cdots\times J_{M} }$   is defined by the operation $\n$ via
\begin{equation*}\label{Eins}
(\mc{A}\n\mc{B})_{i_1...i_Nj_1...j_M}
=\displaystyle\sum_{k_1...k_N}a_{{i_1...i_N}{k_1...k_N}}b_{{k_1...k_N}{j_1...j_M}}.
\end{equation*}
The associative law of this tensor product holds. In the above
formula, if $\mc{B} \in \mathbb{C}^{K_1\times\cdots\times K_N}$,
then $\mc{A}\n\mc{B} \in \mathbb{C}^{I_1\times\cdots\times I_N}$ and
\begin{equation*}
(\mc{A}\n\mc{B})_{i_1...i_N} = \displaystyle\sum_{k_1...k_N}
a_{{i_1...i_N}{k_1...k_N}}b_{{k_1...k_N}}.
\end{equation*}
This product is  used in the study of the theory of relativity
\cite{ein} and in the area of continuum mechanics \cite{lai}. The
Einstein product $\1$  reduces to the standard matrix multiplication
as
$$(A\1B)_{ij}= \displaystyle\sum_{k=1}^{n} a_{ik}b_{kj},$$
for $A\in {\R}^{m \times n}$ and $B\in {\R}^{n \times l}$.

Brazell {\it et al.} \cite{BraliNT13} introduced the notion of the
ordinary tensor inverse,  as follows. A tensor $\mc{X} \in
\mathbb{C}^{I_1\times\cdots\times I_N \times I_1 \times\cdots\times
I_N}$ is called  the  {\it inverse} of $\mc{A}\in
\mathbb{C}^{I_1\times\cdots\times I_N \times I_1 \times\cdots\times
I_N}$  if it satisfies $\mc{A}\n\mc{X}=\mc{X}\n\mc{A}=\mc{I}$. It is
denoted  by $\mc{A}^{-1}$. Sun {\it et al.} \cite{sun} introduced
the Moore-Penrose  generalized inverse of a tensor and its
definition is recalled next.

\begin{definition}(Definition 2.2, \cite{sun})\label{defmpi}\\
Let $\mc{A} \in \mathbb{C}^{I_{1}\times\cdots\times I_{N} \times
J_{1} \times \cdots \times J_{M}}$. The tensor $\mc{X} \in
\mathbb{C}^{J_{1}\times\cdots\times J_{M} \times I_1
\times\cdots\times I_{N}}$ satisfying the following four tensor
equations:

\vspace{-1cm}

\begin{eqnarray}
\mc{A}\m\mc{X}\n\mc{A} &=& \mc{A};\\
\mc{X}\n\mc{A}\m\mc{X} &=& \mc{X};\\
(\mc{A}\m\mc{X})^{H} &=& \mc{A}\m\mc{X};\\
(\mc{X}\n\mc{A})^{H} &=& \mc{X}\n\mc{A},
\end{eqnarray}
is called  the \textbf{Moore-Penrose inverse} of $\mc{A}$, and is
denoted by $\mc{A}^{\dg}$.
\end{definition}
The authors of \cite{bm,  jiw1} further studied different
generalized inverses of tensors via the Einstein product.  Jin {\it
et al.} \cite{bai} again  introduced the Moore-Penrose inverse of  a
tensor using $t-$product. They showed the existence and uniqueness
of the Moore-Penrose inverse of an arbitrary tensor by using the
technique of fast Fourier transform, and discussed an application to
linear models.   Ji and Wei \cite{jiw1} introduced the weighted
Moore-Penrose inverse of an even-order tensor, and again  the
Drazin inverse of an even-order tensor \cite{jiw2}. They \cite{jiw2}
obtained an expression of  the Drazin inverse through the
core-nilpotent decomposition. Applications to find the Drazin
inverse solution of the singular linear tensor equation  $\mc{A} \n
\mc{X} = \mc{B}$ is also presented.  Many results on the generalized
inverses, the $\mc{X}$ which only satisfies some of the four
equations of the Definition \ref{defmpi}, can be found in \cite{bm,
kbm}. The vast work on  generalized inverses of a tensor and its
several multivariety extensions in different areas of mathematics in
the literature, motivate us to study further on theory of
generalized inverses of a tensor. In this paper, we introduce a new
type generalized inverse of tensor via the Einstein product called
\textit{product Moore-Penrose inverse} of tensor. This  is obtained
by extending the defining equations for the Moore-Penrose inverse of
a tensor. In addition, we present various expressions for the
Moore-Penrose inverse of products of tensors.

The paper is outlined as follows. In the next section, we discuss
some notations and definitions which are helpful in proving the main
results.  Section 3 contains the main results.

\section{Preliminaries}
For convenience, we first briefly explain some of the terminologies.
We  refer to ${\C}^{m \times n}$ as the  set of all complex ${m
\times n}$ matrices, where $\C$ denotes the set of complex scalars.
We denote ${\C}^{I_1\times\cdots\times I_N}$  as the set of order
$N$  complex tensors. Indeed, a matrix is a second order tensor, and
a vector is a first order tensor. A tensor $\mc{O}$ denotes the {\it
zero tensor} if  all the entries are zero. A tensor
$\mc{A}\in\mathbb{C}^{I_1\times\cdots\times I_N \times I_1
\times\cdots\times I_N}$ is {\it Hermitian}  if  $\mc{A}=\mc{A}^{H}$
and {\it skew-Hermitian} if $\mc{A}= - \mc{A}^{H}$. A tensor
$\mc{A}\in \mathbb{C}^{I_1\times\cdots\times I_N \times
I_1\times\cdots\times I_N}$  is  {\it idempotent}  if $\mc{A}\n
\mc{A}= \mc{A}.$ Next, we present a result for a Hermitian tensor
which is useful while proving our main results.
\begin{lemma}\label{ir}
If  $\mc{P}=\mc{P}^{H}$ for $\mc{P}\in {\C}^{I_{1}\times\cdots\times
I_{N}\times I_{1}\times \cdots \times I_{N}}$, then for any
$\mc{Q}\in {\C}^{I_{1}\times\cdots\times I_{N}\times I_{1}\times
\cdots \times I_{N}}$,
\begin{eqnarray}
    \mc{P}\n\mc{Q}=\mc{Q}&\text{ implies }& \mc{Q}^{\dg}\n\mc{P}=\mc{Q}^{\dg}.\label{ir1}\\
    \mc{Q}\n\mc{P}=\mc{Q}&\text{ implies }& \mc{P}\n\mc{Q}^{\dg}=\mc{Q}^{\dg}.\label{ir2}
\end{eqnarray}
\end{lemma}
\begin{proof}
Suppose $\mc{P}=\mc{P}^{H}$, and for any $\mc{Q}$,
$\mc{P}\n\mc{Q}=\mc{Q}$. Then $\mc{P}\n\mc{Q}=\mc{Q}$ leads to
$\mc{Q}\n\mc{Q}^{\dg}\n\mc{P}\n\mc{Q}=\mc{Q}$,
$\mc{Q}^{\dg}\n\mc{P}\n\mc{Q}\n\mc{Q}^{\dg}\n\mc{P}=\mc{Q}^{\dg}\n\mc{P}$,
and
$(\mc{Q}^{\dg}\n\mc{P}\n\mc{Q})^{H}=\mc{Q}^{\dg}\n\mc{P}\n\mc{Q}$.
$\mc{P}=\mc{P}^{H}$ and $\mc{P}\n\mc{Q}=\mc{Q}$ leads
$(\mc{Q}\n\mc{Q}^{\dg}\n\mc{P})^{H}=\mc{Q}\n\mc{Q}^{\dg}\n\mc{P}$.
Thus, by Definition \ref{defmpi} first part of the lemma follows.

The other part of the lemma follows by a similar proof.
\end{proof}
The next result states that the Moore-Penrose inverse coincides with
the tensor itself, for a Hermitian and idempotent tensor $\mc{P}$.
\begin{lemma}
  If $\mc{P}\in{\C}^{I_{1}\times\cdots\times I_{N}\times I_{1}\times\cdots\times I_{N}}$ is a Hermitian idempotent tensor, then $\mc{P}^{\dg}=\mc{P}$.
\end{lemma}

Suppose, both $\mc{X}_{1}$ and $\mc{X}_{2}$ are solutions of
$\mc{X}\n\mc{A}\m\mc{X}=\mc{X}$, $\mc{A}\m\mc{X}=\mc{Z}_{1}$, and
$\mc{X}\n\mc{A}=\mc{Z}_{2}$, where $\mc{Z}_{1}$ and $\mc{Z}_{2}$ are
given tensors, independent of $\mc{X}$. Then
$\mc{X}_{1}=\mc{X}_{1}\n\mc{A}\m\mc{X}_{1}$ together with
$\mc{A}\m\mc{X}_{1}=\mc{Z}_{1}=\mc{A}\m\mc{X}_{2}$ and
$\mc{X}_{1}\n\mc{A}=\mc{Z}_{2}=\mc{X}_{2}\n\mc{A}$ results
$\mc{X}_{1}=\mc{X}_{2}$. The following lemma represents this result.

\begin{lemma}\label{ul}
For every $\mc{A}\in{\C}^{I_{1}\times\cdots\times I_{N}\times
J_{1}\times\cdots\times J_{M}}$ there is a unique
$\mc{X}\in{\C}^{J_{1}\times\cdots\times J_{M}\times
I_{1}\times\cdots\times I_{N}}$ satisfying
$\mc{X}\n\mc{A}\m\mc{X}=\mc{X}$, $\mc{A}\m\mc{X}=\mc{Z}_{1}$, and
$\mc{X}\n\mc{A}=\mc{Z}_{2}$, where
$\mc{Z}_{1}\in{\C}^{I_{1}\times\cdots\times I_{N}\times
I_{1}\times\cdots\times I_{N}}$ and
$\mc{Z}_{2}\in{\C}^{J_{1}\times\cdots\times J_{M}\times
J_{1}\times\cdots\times J_{M}}$ are given tensors, independent of
$\mc{X}$.
\end{lemma}


\section{Main Results}
In this section, we first obtain a result to find the Moore-Penrose
inverse of the product of two tensors. We next show the existence
and uniqueness of a new type generalized inverse, i.e., the
\textit{product Moore-Penrose inverse} of a tensor.  We then propose
a few necessary and sufficient conditions for coincidence this
product Moore-Penrose inverse with the Moore-Penrose inverses of
tensors.

Let $\mc{S}\in {\C}^{I_{1}\times \cdots\times I_{N}\times
J_{1}\cdots \times J_{M}}$ and $\mc{T}\in {\C}^{J_{1}\times
\cdots\times J_{M}\times K_{1}\times\cdots \times K_{L}}$ be any two
tensors. Let $\mc{S}_{1}=\mc{S}\m\mc{T}_{1}\lp\mc{T}_{1}^{\dg}$,
where $\mc{T}_{1}=\mc{S}^{\dg}\n\mc{S}\m\mc{T}$. Then
$$\mc{S}_{1}\m\mc{T}_{1}=\mc{S}\m\mc{T}_{1}=\mc{S}\m\mc{T}.$$
Using this factorization, we obtain a representation for
$(\mc{S}\m\mc{T})^{\dg}$ in the following result.
%
\begin{theorem}
Let $\mc{S}\in{\C}^{I_{1}\times\cdots \times I_{N}\times
J_{1}\times\cdots \times J_{M}}$, $\mc{T}\in{\C}^{J_{1}\times\cdots
\times J_{M}\times K_{1}\times\cdots \times K_{L}}$ be any two
tensors. Let $\mc{S}_{1}=\mc{S}\m\mc{T}_{1}\lp\mc{T}_{1}^{\dg}$,
where $\mc{T}_{1}=\mc{S}^{\dg}\n\mc{S}\m\mc{T}$. Then
$$(\mc{S}\m\mc{T})^{\dg}=\mc{T}_{1}^{\dg}\m\mc{S}_{1}^{\dg}.$$
\end{theorem}

\begin{proof}
Let $\mc{A}=\mc{S}\m\mc{T}=\mc{S}_{1}\m\mc{T}_{1}$ and
$\mc{X}=\mc{T}_{1}^{\dg}\m\mc{S}_{1}^{\dg}$. It is sufficient to
show that $\mc{A}$ and $\mc{X}$ satisfy the equations in Definition
\ref{defmpi}. Using the definition of $\mc{S}_{1}$, we have
$\mc{S}_{1}\m\mc{T}_{1}\lp\mc{T}_{1}^{\dg}=\mc{S}_{1}$, which yields
$\mc{A}\lp\mc{X}=\mc{S}_{1}\m\mc{S}_{1}^{\dg}$. Thus,
$\mc{A}\lp\mc{X}\n\mc{A}=\mc{A}$, $\mc{X}\n\mc{A}\lp\mc{X}=\mc{X}$,
and also it follows that $\mc{A}\lp\mc{X}$ is Hermitian. Now, it
remains to show that $\mc{X}\n\mc{A}$ is Hermitian. We observe that
\begin{equation}\label{fteq1}
    \mc{S}^{\dg}\n\mc{S}_{1}=\mc{T}_{1}\lp\mc{T}_{1}^{\dg}.
\end{equation}
Also,
$\mc{S}_{1}^{\dg}\n\mc{S}_{1}\m\mc{T}_{1}\lp\mc{T}_{1}^{\dg}=\mc{S}_{1}^{\dg}\n\mc{S}_{1}$
and
$\mc{T}_{1}\lp\mc{T}_{1}^{\dg}\m\mc{S}_{1}^{\dg}\n\mc{S}_{1}=\mc{S}_{1}^{\dg}\n\mc{S}_{1}$
are equivalent due to the fact that both
$\mc{S}_{1}^{\dg}\n\mc{S}_{1}$ and  $\mc{T}_{1}\lp\mc{T}_{1}^{\dg}$
are Hermitian. Using equation \eqref{fteq1}, we now get
$\mc{S}_{1}^{\dg}\n\mc{S}_{1}=\mc{T}_{1}\lp\mc{T}_{1}^{\dg}.$

From this it follows $\mc{X}\n\mc{A}=\mc{T}_{1}^{\dg}\m\mc{T}_{1}$
is Hermitian. Hence,
$(\mc{S}\m\mc{T})^{\dg}=\mc{T}_{1}^{\dg}\m\mc{S}_{1}^{\dg}.$

\end{proof}

 Existence and uniqueness of a new type of generalized inverse of tensor is proposed in next theorem.

\begin{theorem}\label{Tpgi}
For any tensor $\mc{A}\in {\C}^{I_{1}\times \cdots\times I_{N}\times
J_{1}\times\cdots \times J_{M}}$, let
$\mc{A}=\mc{R}\kp\mc{S}\lp\mc{T}$, where $\mc{R}\in
{\C}^{I_{1}\times \cdots\times I_{N}\times H_{1}\times\cdots \times
H_{K}}$, $\mc{S}\in {\C}^{H_{1}\times \cdots\times H_{K}\times
G_{1}\times\cdots \times G_{L}}$, and $\mc{T}\in
{\C}^{G_{1}\times\cdots \times G_{L}\times J_{1}\times\cdots \times
J_{M}}$. Then, there is a unique matrix $\mc{X}\in
{\C}^{J_{1}\times\cdots \times J_{M}\times I_{1}\times \cdots\times
I_{N}}$ such that
\begin{eqnarray}
\mc{A}\m\mc{X}\n\mc{A}&=&\mc{A,}\label{pgi1}\\
\mc{X}\n\mc{A}\m\mc{X}&=&\mc{X},\label{pgi2}\\
(\mc{R}^{\dg}\n\mc{A}\m\mc{X}\n\mc{R})^{H}&=& \mc{R}^{\dg}\n\mc{A}\m\mc{X}\n\mc{R},\label{pgi3}\\
(\mc{T}\m\mc{X}\n\mc{A}\m\mc{T}^{\dg})^{H}&=& \mc{T}\m\mc{X}\n\mc{A}\m\mc{T}^{\dg},\label{pgi4}\\
\mc{X}\n\mc{R}\kp\mc{R}^{\dg}&=&\mc{X},\label{pgi5}\\
\mc{T}^{\dg}\lp\mc{T}\n\mc{X}&=&\mc{X}\label{pgi6}.
\end{eqnarray}
\end{theorem}
\begin{proof}
$\mc{A}=\mc{R}\kp\mc{S}\lp\mc{T}$ gives
\begin{equation}\label{pgipe1}
    \mc{R}\kp\mc{R}^{\dg}\n\mc{A}=\mc{A}=\mc{A}\m\mc{T}^{\dg}\lp\mc{T}.
\end{equation}
Equations \eqref{pgi1}, \eqref{pgi2}, \eqref{pgi3} and \eqref{pgi4}
together with equation \eqref{pgipe1} yield,
$(\mc{R}^{\dg}\n\mc{A}\m\mc{T}^{\dg})\lp(\mc{T}\m\mc{X}$
$\n\mc{R})\kp(\mc{R}^{\dg}\n\mc{A}\m\mc{T}^{\dg})=\mc{R}^{\dg}\n\mc{A}\m\mc{T}^{\dg}$,
$(\mc{T}\m\mc{X}\n\mc{R})\kp(\mc{R}^{\dg}\n\mc{A}\m\mc{T}^{\dg})\lp(\mc{T}\m\mc{X}\n\mc{R})=(\mc{T}$
$\n\mc{X}\n\mc{R})$,
$[(\mc{R}^{\dg}\n\mc{A}\m\mc{T}^{\dg})\lp(\mc{T}\m\mc{X}\n\mc{R})]^{H}=(\mc{R}^{\dg}\n\mc{A}\m\mc{T}^{\dg})\lp(\mc{T}\m\mc{X}\n\mc{R})$,
and $[(\mc{T}$
$\m\mc{X}\n\mc{R})\kp(\mc{R}^{\dg}\n\mc{A}\m\mc{T}^{\dg})]^{H}=(\mc{T}\m\mc{X}\n\mc{R})\kp(\mc{R}^{\dg}\n\mc{A}\m\mc{T}^{\dg})$,
respectively. We thus have
\begin{equation}\label{pgi}
    (\mc{R}^{\dg}\n\mc{A}\m\mc{T}^{\dg})^{\dg}=\mc{T}\m\mc{X}\n\mc{R}.
\end{equation}
Conversely, suppose equation \eqref{pgi} holds. Using equation
\eqref{pgipe1},
$\mc{A}\m\mc{X}\n{A}=\mc{R}\kp\mc{R}^{\dg}\n\mc{A}\m$
$\mc{T}^{\dg}\lp\mc{T}=\mc{A}$,
$(\mc{R}^{\dg}\n\mc{A}\m\mc{X}\n\mc{R})^{H}=\mc{R}^{\dg}\n\mc{A}\m\mc{T}^{\dg}\lp\mc{T}\m\mc{X}\n\mc{R}=\mc{R}^{\dg}\n\mc{A}\m\mc{X}\n\mc{R}$,
$(\mc{T}\m$
$\mc{X}\n\mc{A}\m\mc{T}^{\dg})^{H}=\mc{T}\m\mc{X}\n\mc{R}\kp\mc{R}^{\dg}\n\mc{A}\m\mc{T}^{\dg}=\mc{T}\m\mc{X}\n\mc{A}\m\mc{T}^{\dg}$.
But, if equation \eqref{pgi5} and \eqref{pgi6} hold, then using
equation \eqref{pgipe1}
\begin{align*}
    \mc{X}\n\mc{A}\m\mc{X}&=\mc{T}^{\dg}\lp\mc{T}\m\mc{X}\n\mc{A}\m\mc{X}\n\mc{R}\kp\mc{R}^{\dg}\\
    &=\mc{T}^{\dg}\lp\mc{T}\m\mc{X}\n\mc{R}\kp\mc{R}^{\dg}\\
    &=\mc{X}.
\end{align*}
Thus, equations \eqref{pgi1}-\eqref{pgi6} is equivalent to equations
\eqref{pgi5}, \eqref{pgi6} and \eqref{pgi}. And from the latter
three equations
\begin{align}
    \mc{X}&=\mc{T}^{\dg}\lp\mc{T}\m\mc{X}\n\mc{R}\kp\mc{R}^{\dg}\nonumber\\
    &=\mc{T}^{\dg}\lp(\mc{R}^{\dg}\n\mc{A}\m\mc{T}^{\dg})^{\dg}\kp\mc{R}^{\dg}\label{epgi}.
\end{align}
Suppose there exist $\mc{X}$ and $\mc{Y}$ satisfying \eqref{pgi5},
\eqref{pgi6} and \eqref{pgi}, then
\begin{align*}
    \mc{X}&=\mc{T}^{\dg}\lp\mc{T}\m\mc{X}\n\mc{R}\kp\mc{R}^{\dg}\\
    &=\mc{T}^{\dg}\lp(\mc{R}^{\dg}\n\mc{A}\m\mc{T}^{\dg})^{\dg}\kp\mc{R}^{\dg}\\
    &=\mc{T}^{\dg}\lp\mc{T}\m\mc{Y}\n\mc{R}\kp\mc{R}^{\dg}\\
    &=\mc{Y}.
\end{align*}
Hence the uniqueness is established.
\end{proof}
 We termed the tensor $\mc{X}$ in the equation \eqref{epgi} as the \textit{product Moore-Penrose inverse} of $\mc{A}$, and denote it as $\mc{A}_{\pi \dg}$.
An alternative representation of the  product Moore-Penrose inverse
is given in the following theorem.
\begin{theorem}\label{pgif}
Let $\mc{A}=\mc{R}\kp\mc{S}\lp\mc{T}\in {\C}^{I_{1}\times
\cdots\times I_{N}\times J_{1}\times\cdots \times J_{M}}$, where
$\mc{R}\in {\C}^{I_{1}\times \cdots\times I_{N}\times
H_{1}\times\cdots \times H_{K}}$, $\mc{S}\in {\C}^{H_{1}\times
\cdots\times H_{K}\times G_{1}\times\cdots \times G_{L}}$, and
$\mc{T}\in {\C}^{G_{1}\times\cdots \times G_{L}\times
J_{1}\times\cdots \times J_{M}}$. Then
$$(\mc{R}^{\dg}\n\mc{A}\m\mc{T}^{\dg})^{\dg}=(\mc{A}\m\mc{T}^{\dg})^{\dg}\n\mc{A}\m(\mc{R}^{\dg}\n\mc{A})^{\dg}.$$

\end{theorem}
\begin{proof}
In view of equation \eqref{epgi}, equation \eqref{pgi1} results
\begin{equation}\label{apgie1}
\mc{A}=\mc{A}\m\mc{T}^{\dg}\lp(\mc{R}^{\dg}\n\mc{A}\m\mc{T}^{\dg})^{\dg}\kp\mc{R}^{\dg}\n\mc{A}.
\end{equation}
Pre-multiplying $(\mc{A}\m\mc{T}^{\dg})^{\dg}$ and post-multiplying
$(\mc{R}^{\dg}\n\mc{A})^{\dg}$ to equation \eqref{apgie1}, we obtain
\begin{equation}
(\mc{A}\m\mc{T}^{\dg})^{\dg}\n\mc{A}\m(\mc{R}^{\dg}\n\mc{A})^{\dg}=(\mc{R}^{\dg}\n\mc{A}\m\mc{T}^{\dg})^{\dg}.
\end{equation}
\end{proof}

One more expression of $\mc{A}_{\pi\dg}$ is provided below,
\begin{theorem}\label{aepgi}
$\mc{A}_{\pi\dg}=\mc{T}^{\dg}\lp(\mc{A}\m\mc{T}^{\dg})^{\dg}\n\mc{A}\m(\mc{R}^{\dg}\n\mc{A})^{\dg}\kp\mc{R}^{\dg}$
is the unique tensor $\mc{X}$, such that
$\mc{X}\n\mc{A}\m\mc{X}=\mc{X}$,
$\mc{A}\m\mc{X}=\mc{A}\m(\mc{R}^{\dg}\n\mc{A})^{\dg}\kp\mc{R}^{\dg}$
and
$\mc{X}\n\mc{A}=\mc{T}^{\dg}\lp(\mc{A}\m\mc{T}^{\dg})^{\dg}\n\mc{A}$.
\end{theorem}
\begin{proof}
By using Theorem \ref{pgif}z and equation \eqref{epgi}, we have
$$\mc{X}=\mc{A}_{\pi\dg}=\mc{T}^{\dg}\lp(\mc{A}\m\mc{T}^{\dg})^{\dg}\n\mc{A}\m(\mc{R}^{\dg}\n\mc{A})^{\dg}\kp\mc{R}^{\dg}.$$
A simple calculation leads to $\mc{X}\n\mc{A}\m\mc{X}=\mc{X}$. Since
$\mc{R}\kp\mc{R}^{\dg}\n\mc{A}=\mc{A}=\mc{A}\m\mc{T}^{\dg}\lp\mc{T}$,
so
$\mc{A}\m\mc{X}=\mc{A}\m(\mc{R}^{\dg}\n\mc{A})^{\dg}\kp\mc{R}^{\dg}$
and
$\mc{X}\n\mc{A}=\mc{T}^{\dg}\lp(\mc{A}\m\mc{T}^{\dg})^{\dg}\n\mc{A}$.
The uniqueness follows from Lemma \ref{ul}.

\end{proof}
Next, we present a result to find the Moore-Penrose inverse of a
tensor $\mc{A}$ with the factorization
$\mc{A}=\mc{R}\kp\mc{S}\lp\mc{T}$.

\begin{theorem}\label{aegi}Let $\mc{A}=\mc{R}\kp\mc{S}\lp\mc{T}\in {\C}^{I_{1}\times \cdots\times I_{N}\times J_{1}\times\cdots \times J_{M}}$, where $\mc{R}\in {\C}^{I_{1}\times \cdots\times I_{N}\times H_{1}\times\cdots \times H_{K}}$, $\mc{S}\in {\C}^{H_{1}\times \cdots\times H_{K}\times G_{1}\times\cdots \times G_{L}}$, and $\mc{T}\in {\C}^{G_{1}\times\cdots \times G_{L}\times J_{1}\times\cdots \times J_{M}}$. Then
\begin{align*}\mc{A}^{\dg}=(\mc{R}\kp\mc{S}\lp\mc{T})^{\dg}&=(\mc{R}^{\dg}\n\mc{A})^{\dg}\kp\mc{R}^{\dg}\n\mc{A}\m\mc{T}^{\dg}\lp(\mc{A}\m\mc{T}^{\dg})^{\dg}\\
&=(\mc{R}^{\dg}\n\mc{A})^{\dg}\kp\mc{S}\lp(\mc{A}\m\mc{T}^{\dg})^{\dg}.
\end{align*}

\end{theorem}
\begin{proof}
Using equation \eqref{pgipe1},
\begin{align}
    \mc{A}\m(\mc{R}^{\dg}\n\mc{A})^{\dg}\kp(\mc{R}^{\dg}\n\mc{A})&=\mc{A},\label{ea1}\\
    (\mc{A}\m\mc{T}^{\dg})\lp(\mc{A}\m\mc{T}^{\dg})^{\dg}\n\mc{A}&=\mc{A}.\label{ea2}
\end{align}
Pre-multiplying equation \eqref{ea1} and post-multiplying equation
\eqref{ea2} by $\mc{A}^{\dg}$, and using Lemma \ref{ir}, we have
\begin{align}
    (\mc{R}^{\dg}\n\mc{A})^{\dg}\kp\mc{R}^{\dg}\n\mc{A}=\mc{A}^{\dg}\n\mc{A},\label{Ea+a}\\
    \mc{A}\m\mc{T}^{\dg}\lp(\mc{A}\m\mc{T}^{\dg})^{\dg}=\mc{A}\m\mc{A}^{\dg},\label{Eaa+}
\end{align}
respectively, since $\mc{A}^{\dg}\n\mc{A}$ and
$\mc{A}\m\mc{A}^{\dg}$ are Hermitian.
$\mc{A}^{\dg}=\mc{A}^{\dg}\n\mc{A}\m\mc{A}^{\dg}$ with the last two
expressions yields
$$\mc{A}^{\dg}=(\mc{R}\kp\mc{S}\lp\mc{T})^{\dg}=(\mc{R}^{\dg}\n\mc{A})^{\dg}\kp\mc{R}^{\dg}\n\mc{A}\m\mc{T}^{\dg}\lp(\mc{A}\m\mc{T}^{\dg})^{\dg}.$$
And
$\mc{A}^{\dg}=(\mc{R}\kp\mc{S}\lp\mc{T})^{\dg}=(\mc{R}^{\dg}\n\mc{A})^{\dg}\kp\mc{S}\lp(\mc{A}\m\mc{T}^{\dg})^{\dg},$
is due to Lemma \ref{ir}.
\end{proof}
 The following corollary immediately follows from the above result when $\mc{S}=\mc{I}\in{\C}^{J_{1}\times\cdots\times J_{M}\times J_{1}\times \cdots\times J_{M}}$.
\begin{corollary}\label{caegi}
For any tensor $\mc{M}\in{\C}^{I_{1}\times\cdots\times I_{N}\times
J_{1}\times \cdots\times J_{M}}$ and
$\mc{N}\in{\C}^{J_{1}\times\cdots\times J_{M}\times H_{1}\times
\cdots\times H_{K}}$,
$(\mc{M}\m\mc{N})^{\dg}=(\mc{M}^{\dg}\n\mc{M}\m\mc{N})^{\dg}\n(\mc{M}\m\mc{N}\kp\mc{N}^{\dg})^{\dg}$.
\end{corollary}
\begin{proof} This follows from Theorem \ref{aegi}, by setting $\mc{R}=\mc{M}$, $\mc{S}=\mc{I}\in{\C}^{J_{1}\times\cdots\times J_{M}\times J_{1}\times \cdots\times J_{M}}$, $\mc{T}=\mc{N}$ and $\mc{A}=\mc{M}\n\mc{N}$.
\end{proof}

We observed that $\mc{P}^{\dg}=\mc{P}$ for a Hermitian idempotent
tensor $\mc{P}$. Combining this fact with the expression for
$(\mc{M}\n\mc{N})^{\dg}$ in Corollary \ref{caegi} gives the
following result.
\begin{corollary}
Let $\mc{M}\in{\C}^{I_{1}\times\cdots\times I_{N}\times I_{1}\times
\cdots\times I_{N}}$ and $\mc{N}\in{\C}^{I_{1}\times\cdots\times
I_{N}\times I_{1}\times \cdots\times I_{N}}$ be any two Hermitian
idempotent tensors, then $(\mc{M}\n\mc{N})^{\dg}$ is idempotent.
\end{corollary}


\subsection{Relation between $\mc{A}_{\pi\dg}$ and $\mc{A}^{\dg}$}

Let $\mc{A}=\mc{R}\kp\mc{S}\lp\mc{T}\in {\C}^{I_{1}\times
\cdots\times I_{N}\times J_{1}\times\cdots \times J_{M}}$, where
$\mc{R}\in {\C}^{I_{1}\times \cdots\times I_{N}\times
H_{1}\times\cdots \times H_{K}}$, $\mc{S}\in {\C}^{H_{1}\times
\cdots\times H_{K}\times G_{1}\times\cdots \times G_{L}}$ and
$\mc{T}\in {\C}^{G_{1}\times\cdots \times G_{L}\times
J_{1}\times\cdots \times J_{M}}$. Let us also denote
$\mc{T}^{\dg}\lp(\mc{A}\m\mc{T}^{\dg})^{\dg}$ and
$(\mc{R}^{\dg}\n\mc{A})^{\dg}\kp\mc{R}^{\dg}$ as $\mc{B}$ and
$\mc{C}$, respectively. With these notations, Theorem \ref{aepgi}
and Theorem \ref{aegi} result
\begin{equation}\label{Eepgi}
    \mc{A}_{\pi \dg}=\mc{B}\n\mc{A}\m\mc{C},
\end{equation}
 and \begin{equation}\label{Egi}
  ~~ \mc{A}^{\dg}=\mc{C}\n\mc{A}\m\mc{B}.
\end{equation}
Also,
\begin{equation}\label{eu}
    \mc{A}_{\pi\dg}\n\mc{A}\m\mc{A}^{\dg}=\mc{B},
\end{equation}
and
\begin{equation}\label{ev}
\mc{A}^{\dg}\n\mc{A}\m\mc{A}_{\pi\dg}=\mc{C},
\end{equation}
follow from Theorem \ref{aepgi} and Lemma \ref{ir}. In the next
result, we represent $\mc{B}$ and $\mc{C}$ as product Moore-Penrose
inverse.
\begin{theorem}\label{Tapgi}
Let $\mc{A}=\mc{R}\kp\mc{S}\lp\mc{T}\in {\C}^{I_{1}\times
\cdots\times I_{N}\times J_{1}\times\cdots \times J_{M}}$, where
$\mc{R}\in {\C}^{I_{1}\times \cdots\times I_{N}\times
H_{1}\times\cdots \times H_{K}}$, $\mc{S}\in {\C}^{H_{1}\times
\cdots\times H_{K}\times G_{1}\times\cdots \times G_{L}}$ and
$\mc{T}\in {\C}^{G_{1}\times\cdots \times G_{L}\times
J_{1}\times\cdots \times J_{M}}$.
\begin{itemize}
\item[(i)] If $\mc{A}=\mc{A}\m(\mc{A}^{\dg}\n\mc{A}\m\mc{T}^{\dg})\lp\mc{T}$, then $\mc{A}_{\pi\dg}=\mc{B}$.
\item[(ii)] If $\mc{A}=\mc{R}\kp(\mc{R}^{\dg}\n\mc{A}\m\mc{A}^{\dg})\n\mc{A}$, then $\mc{A}_{\pi\dg}=\mc{C}$.
\end{itemize}
\end{theorem}
\begin{proof} $(i)$ Applying Theorem \ref{aepgi} to $\mc{A}\m(\mc{A}^{\dg}\n\mc{A}\m\mc{T}^{\dg})\lp\mc{T}=\mc{A}\m\mc{T}^{\dg}\lp\mc{T}$ results
\begin{multline*}
[\mc{A}\m(\mc{A}^{\dg}\n\mc{A}\m\mc{T}^{\dg})\lp\mc{T}]_{\pi\dg}=[\mc{A}\m\mc{T}^{\dg}\lp\mc{T}]_{\pi\dg}\\
=\mc{T}^{\dg}\lp(\mc{A}\m\mc{T}^{\dg}\lp\mc{T}\m\mc{T}^{\dg})^{\dg}\n\mc{A}\\
\m\mc{T}^{\dg}\lp\mc{T}\m(\mc{A}^{\dg}\n\mc{A}\m\mc{T}^{\dg}\lp\mc{T})^{\dg}\m\mc{A}^{\dg},
\end{multline*}
which reduces to
\begin{eqnarray*}
[\mc{A}\m(\mc{A}^{\dg}\n\mc{A}\m\mc{T}^{\dg})\lp\mc{T}]_{\pi\dg}&=&\mc{T}^{\dg}\lp(\mc{A}\m\mc{T}^{\dg})^{\dg}\n\mc{A}\m(\mc{A}^{\dg}\n\mc{A})^{\dg}\m\mc{A}^{\dg},
\end{eqnarray*}
due to equation \eqref{pgipe1}. Since $\mc{A}^{\dg}\n\mc{A}$ is
Hermitian, we have
$$[\mc{A}\m(\mc{A}^{\dg}\n\mc{A}\m\mc{T}^{\dg})\lp\mc{T}]_{\pi\dg}=\mc{T}^{\dg}\lp(\mc{A}\m\mc{T}^{\dg})^{\dg}=\mc{B},$$
by Lemma \ref{ir}.\\
$(ii)$ Similarly, we have
$\mc{C}=[\mc{R}\kp(\mc{R}^{\dg}\n\mc{A}\m\mc{A}^{\dg})\n\mc{A}]_{\pi\dg}.$
\end{proof}
Then using  Theorem \ref{Tpgi} and Theorem \ref{Tapgi}, it is easy
to show that $\mc{B}\n\mc{A}\m\mc{B}=\mc{B}$ and
$\mc{C}\n\mc{A}\m\mc{C}=\mc{C}$.
We next state two results on $\mc{B}$ and $\mc{C}$ with the help of
the expressions for
$\mc{A}\m\mc{B},~\mc{B}\n\mc{A},~\mc{A}\m\mc{C},$ and
$\mc{C}\n\mc{A}$ obtained from equation \eqref{eu} and equation
\eqref{ev}.
\begin{remark}\label{CaepgiU}
The system of tensor equations $\mc{X}\n\mc{A}\m\mc{X}=\mc{X}$,
$\mc{A}\m\mc{X}=\mc{A}\m\mc{A}^{\dg}$ and
$\mc{X}\n\mc{A}=\mc{A}_{\pi \dg}\n\mc{A}$
($\mc{A}\m\mc{X}=\mc{A}\m\mc{A}_{\pi \dg}$ and
$\mc{X}\n\mc{A}=\mc{A}^{\dg}\n\mc{A}$) has unique solution
$\mc{X}=\mc{B}$ $(\mc{X}=\mc{C})$.
\end{remark}

 In terms of $\mc{B}$ and $\mc{C}$, Theorem \ref{aepgi}, the equations \eqref{Ea+a} and \eqref{Eaa+}  can be rewritten as following.
\begin{remark}
$\mc{X}=\mc{A}^{\dg}(\mc{X}=\mc{A}_{\pi\dg})$ is the unique matrix
such that $\mc{X}\n\mc{A}\m\mc{X}=\mc{X}$,
$\mc{A}\m\mc{X}=\mc{A}\m\mc{B}$ and $\mc{X}\n\mc{A}=\mc{C}\n\mc{A}$
($\mc{A}\m\mc{X}=\mc{A}\m\mc{C}$ and
$\mc{X}\n\mc{A}=\mc{B}\n\mc{A}$).
\end{remark}

The next example confirms that  any two of
$\mc{A}^{\dg},~\mc{A}_{\pi\dg},~\mc{B}$ and $\mc{C}$ need not
coincide for a given factorization
$\mc{A}=\mc{R}\kp\mc{S}\lp\mc{T}\in {\C}^{I_{1}\times \cdots\times
I_{N}\times J_{1}\times\cdots \times J_{M}}$, where $\mc{R}\in
{\C}^{I_{1}\times \cdots\times I_{N}\times H_{1}\times\cdots \times
H_{K}}$, $\mc{S}\in {\C}^{H_{1}\times \cdots\times H_{K}\times
G_{1}\times\cdots \times G_{L}}$ and $\mc{T}\in
{\C}^{G_{1}\times\cdots \times G_{L}\times J_{1}\times\cdots \times
J_{M}}$.
\begin{example}
Consider the tensor $\mc{A}=(a_{ijkl})_{1\leq
i,j,k,l\leq2}\in{\C}^{2\times2\times2\times 2}$, such that

\vspace{0.5cm} \noindent
$$
a_{ij11}=
        \begin{pmatrix}
              1 & 0\\
              0 & -1
        \end{pmatrix},~
a_{ij12}=
        \begin{pmatrix}
              0 & -1\\
              0 & 0
        \end{pmatrix},~
a_{ij21}=
        \begin{pmatrix}
             0 & 1\\
             0 & -1
        \end{pmatrix},~
a_{ij22}=
        \begin{pmatrix}
            0 & 1\\
            0 & 0
        \end{pmatrix}.
$$

Suppose that $\mc{A}$ is factorized as  $\mc{A}=\mc{R}\2\mc{S}\2
\mc{T}$, where $\mc{R}=(r_{ijkl})\in{\C}^{2\times2\times2\times
2},~\mc{S}=(s_{ijkl})\in{\C}^{2\times2\times2\times
2},~\mc{T}=(t_{ijkl})\in{\C}^{2\times2\times2\times 2},~1\leq
i,j,k,l\leq2$ such that

\vspace{0.5cm} \noindent
$$
r_{ij11}=
         \begin{pmatrix}
                  0 & 1\\
                  1 & 0
         \end{pmatrix},~
r_{ij12}=
         \begin{pmatrix}
                      -1 & 0\\
                       1 & 0
          \end{pmatrix},~
r_{ij21}=
          \begin{pmatrix}
                       0 & 1\\
                       0 & 1
          \end{pmatrix},~
r_{ij22}=
        \begin{pmatrix}
                     0 & 0\\
                     1 & 1
        \end{pmatrix};
$$

\vspace{0.5cm} \noindent
$$
s_{ij11}=
        \begin{pmatrix}
                     1 & 0\\
                     1 & 0
        \end{pmatrix},~
s_{ij12}=
        \begin{pmatrix}
                     0 & -1\\
                     0 & 1
        \end{pmatrix},~
s_{ij21}=
        \begin{pmatrix}
                     1 & 0\\
                     0 & 0
        \end{pmatrix},~
s_{ij22}=
        \begin{pmatrix}
                     0 & 0\\
                     1 & 0
        \end{pmatrix};
$$

\vspace{0.5cm} \noindent and
$$
t_{ij11}=
        \begin{pmatrix}
                     1 & 0\\
                     0 & 0
        \end{pmatrix},~
t_{ij12}=
        \begin{pmatrix}
                     0 & 0\\
                     1 & 0
        \end{pmatrix},~
t_{ij21}=
        \begin{pmatrix}
                     0 & 1\\
                     0 & 0
        \end{pmatrix},~
t_{ij22}=
        \begin{pmatrix}
                     0 & 0\\
                     1 & -1
        \end{pmatrix}.
$$

\vspace{0.5cm} \noindent Now, $\mc{A}^{\dg}=(a'_{ijkl})_{1\leq
i,j,k,l\leq2}\in{\C}^{2\times2\times2\times 2}$,
$\mc{R}^{\dg}=(r'_{ijkl})_{1\leq
i,j,k,l\leq2}\in{\C}^{2\times2\times2\times 2}$,
$\mc{S}^{\dg}=(s'_{ijkl})_{1\leq
i,j,k,l\leq2}\in{\C}^{2\times2\times2\times 2}$, and
$\mc{T}^{\dg}=(t'_{ijkl})_{1\leq
i,j,k,l\leq2}\in{\C}^{2\times2\times2\times 2}$ are given by

\vspace{0.5cm} \noindent
$$
a'_{ij11}=
         \begin{pmatrix}
                      1 & -1\slash 2\\
                      -1 & 1\slash2
        \end{pmatrix},~
a'_{ij12}=
         \begin{pmatrix}
                      0&-1\slash 2\\
                      0 & 1\slash2
        \end{pmatrix},~
a'_{ij21}=
         \begin{pmatrix}
                      0 & 0\\
                      0 & 0
        \end{pmatrix},~
a'_{ij22}=
         \begin{pmatrix}
                      0 & -1\slash 2\\
                      -1 & 1\slash2
        \end{pmatrix};
$$

\vspace{0.5cm} \noindent
$$
r'_{ij11}=
         \begin{pmatrix}
                      1 & -1\\
                      -1 & 1
         \end{pmatrix},~
r'_{ij12}=
         \begin{pmatrix}
                      0 & 0\\
                      1 & -1
        \end{pmatrix},~
r'_{ij21}=
         \begin{pmatrix}
                      1 & 0\\
                      -1 & 1
         \end{pmatrix},~
r'_{ij22}=
         \begin{pmatrix}
                      0 & 0\\
                      0 & 1
         \end{pmatrix};
$$

\vspace{0.5cm} \noindent
$$
s'_{ij11}=
         \begin{pmatrix}
                      1\slash3 & 0\\
                      2\slash3 & -1\slash3
         \end{pmatrix},~
s'_{ij12}=
         \begin{pmatrix}
                      0 & -1\slash2\\
                      0 & 0
        \end{pmatrix},~
s'_{ij21}=
         \begin{pmatrix}
                      1\slash3 & 0\\
                      -1\slash3 & 2\slash3
         \end{pmatrix},~
s'_{ij22}=
         \begin{pmatrix}
                      0 & 1\slash2\\
                      0 & 0
         \end{pmatrix};
$$

\vspace{0.5cm} \noindent and
$$
t'_{ij11}=
         \begin{pmatrix}
                      1 & 0\\
                      0 & 0
         \end{pmatrix},~
t'_{ij12}=
         \begin{pmatrix}
                      0 & 0\\
                      1 & 0
         \end{pmatrix},~
t'_{ij21}=
         \begin{pmatrix}
                      0 & 1\\
                      0 & 1
         \end{pmatrix},~
t'_{ij22}=
         \begin{pmatrix}
                      0 & 1\\
                      0 & -1
         \end{pmatrix}.
$$

\vspace{0.2cm} \noindent
 Here, $\mc{R}^{\dg}\2\mc{A}\2\mc{T}^{\dg}=\mc{S}$, then  $\mc{A}_{\pi\dg}=\mc{T}^{\dg}\2(\mc{R}^{\dg}\2\mc{A}\2\mc{T}^{\dg})^{\dg}\2\mc{R}^{\dg}=\mc{T}^{\dg}\2\mc{S}^{\dg}\2\mc{R}^{\dg}=(x_{ijkl})$,  $\mc{B}=\mc{A}_{\pi\dg}\2\mc{A}\2\mc{A}^{\dg}=(u_{ijkl})$, and $\mc{C}=\mc{A}^{\dg}\2\mc{A}\2\mc{A}_{\pi\dg}=(v_{ijkl})$, where

\vspace{0.5cm} \noindent
 $$
 x_{ij11}=
         \begin{pmatrix}
                      1 & -1\slash3\\
                      -1 & 2\slash3
         \end{pmatrix},~
x_{ij12}=
         \begin{pmatrix}
                      0 & -1\slash3\\
                      0 & 2\slash3
         \end{pmatrix},~
x_{ij21}=
        \begin{pmatrix}
                     0 & 0\\
                     -1\slash2 & 1\slash2
        \end{pmatrix},~
x_{ij22}=
        \begin{pmatrix}
                     0 & 0\\
                     -1 & 1
        \end{pmatrix};
$$

\vspace{0.5cm} \noindent
$$
 u_{ij11}=
         \begin{pmatrix}
                      1 & -1\slash2\\
                      -1 & 1\slash2
         \end{pmatrix},~
u_{ij12}=
        \begin{pmatrix}
                     0 & -1\slash2\\
                     0 & 1\slash2
        \end{pmatrix},
u_{ij21}=
        \begin{pmatrix}
                     0 & -1\slash4\\
                     -1\slash2 & 1\slash4
        \end{pmatrix},~
u_{ij22}=
        \begin{pmatrix}
                     0 & -1\slash2\\
                     -1 & 1\slash2
        \end{pmatrix};
 $$

\vspace{0.5cm} \noindent and
$$
v_{ij11}=
        \begin{pmatrix}
                     1 & -1\slash3\\
                     -1 & 2\slash3
        \end{pmatrix},~
v_{ij12}=
        \begin{pmatrix}
                     0 & -1\slash3\\
                     0 & 2\slash3
        \end{pmatrix},~
v_{ij21}=
        \begin{pmatrix}
                     0 & 0\\
                     0 & 0
        \end{pmatrix},~~and~~
v_{ij22}=
        \begin{pmatrix}
                     0 & 0\\
                     -1 & 1
        \end{pmatrix}.
$$
\end{example}
At this point one may be interested to know when does the
\textit{product Moore-Penrose inverse} coincide with the
\textit{Moore-Penrose inverse}? The answer to this question is
explained in the following theorem.
\begin{theorem}\label{Tnscpgi}
Let $\mc{A}=\mc{R}\kp\mc{S}\lp\mc{T}\in {\C}^{I_{1}\times
\cdots\times I_{N}\times J_{1}\times\cdots \times J_{M}}$, where
$\mc{R}\in {\C}^{I_{1}\times \cdots\times I_{N}\times
H_{1}\times\cdots \times H_{K}}$, $\mc{S}\in {\C}^{H_{1}\times
\cdots\times H_{K}\times G_{1}\times\cdots \times G_{L}}$ and
$\mc{T}\in {\C}^{G_{1}\times\cdots \times G_{L}\times
J_{1}\times\cdots \times J_{M}}$. Then
$\mc{A}_{\pi\dg}=\mc{A}^{\dg}$ if and only if
$(\mc{R}^{\dg}\n\mc{A}\m
\mc{T}^{\dg})^{\dg}=\mc{T}\m\mc{A}^{\dg}\n\mc{R}$. In that case
$\mc{A}_{\pi\dg}=\mc{A}^{\dg}=\mc{B}=\mc{C}$.
\end{theorem}
\begin{proof}
Suppose
$(\mc{R}^{\dg}\n\mc{A}\m\mc{T}^{\dg})^{\dg}=\mc{T}\m\mc{A}^{\dg}\n\mc{R}$,
then by equation \eqref{epgi} and Lemma \ref{ir}, we have
\begin{eqnarray*}
\mc{A}_{\pi\dg}&=&\mc{T}^{\dg}\lp(\mc{R}^{\dg}\n\mc{A}\m\mc{T}^{\dg})^{\dg}\kp\mc{R}^{\dg}\\
&=&\mc{T}^{\dg}\lp\mc{T}\m\mc{A}^{\dg}\n\mc{R}\kp\mc{R}^{\dg}\\
&=&\mc{A}^{\dg}.
\end{eqnarray*}
And using equation \eqref{eu} and \eqref{ev}, we have
$\mc{B}=\mc{A}_{\pi\dg}\n\mc{A}\m\mc{A}^{\dg}=\mc{A}^{\dg}\n\mc{A}\m\mc{A}^{\dg}=\mc{A}^{\dg}=\mc{A}^{\dg}\n\mc{A}\m\mc{A}^{\dg}=\mc{A}^{\dg}\n\mc{A}\m\mc{A}_{\pi\dg}=\mc{C}.$

Conversely, suppose that $\mc{A}_{\pi\dg}=\mc{A}^{\dg}$. By using
equation \eqref{pgipe1} and Lemma \ref{ir}, we then have
\begin{align}
    \mc{T}\m\mc{A}^{\dg}\n\mc{R}&=\mc{T}\m\mc{A}_{\pi\dg}\n\mc{R}\nonumber\\
    &=\mc{T}\m\mc{T}^{\dg}\lp(\mc{R}^{\dg}\n\mc{A}\m\mc{T}^{\dg})^{\dg}\kp\mc{R}^{\dg}\n\mc{R}\label{Enscpgi}\\
    &=\mc{T}\m\mc{T}^{\dg}\lp(\mc{R}^{\dg}\n\mc{A}\m\mc{T}^{\dg})^{\dg}\nonumber\\
    &=(\mc{R}^{\dg}\n\mc{A}\m\mc{T}^{\dg})^{\dg}\nonumber.
\end{align}

\end{proof}

From the proof of Theorem \ref{Tnscpgi}, it is clear that
$\mc{A}_{\pi\dg}=\mc{A}^{\dg}$ implies $\mc{B}=\mc{C}$. However, the
converse is also true. Because if $\mc{B}=\mc{C}$, then the relation
$\mc{B}\n\mc{A}\m\mc{B}=\mc{B}$ with equation \eqref{Eepgi} and
equation \eqref{Egi} results $\mc{B}=\mc{A}_{\pi\dg}=\mc{A}^{\dg}$.
Thus, Theorem \ref{Tnscpgi} also can be restated as follows.

\begin{remark}
Let $\mc{A}=\mc{R}\kp\mc{S}\lp\mc{T}\in {\C}^{I_{1}\times
\cdots\times I_{N}\times J_{1}\times\cdots \times J_{M}}$, where
$\mc{R}\in {\C}^{I_{1}\times \cdots\times I_{N}\times
H_{1}\times\cdots \times H_{K}}$, $\mc{S}\in {\C}^{H_{1}\times
\cdots\times H_{K}\times G_{1}\times\cdots \times G_{L}}$ and
$\mc{T}\in {\C}^{G_{1}\times\cdots \times G_{L}\times
J_{1}\times\cdots \times J_{M}}$. Then,
$\mc{A}_{\pi\dg}=\mc{A}^{\dg}$ if and only if $\mc{B}=\mc{C}$, in
which case $\mc{A}_{\pi\dg}=\mc{A}^{\dg}=\mc{B}=\mc{C}$.
\end{remark}

One may be interested on the case that either of $\mc{B}$ or
$\mc{C}$ coincides with $\mc{A}^{\dg}$ but $\mc{B}\neq\mc{C}$. The
answer is shown below.
\begin{theorem}
Let $\mc{A}=\mc{R}\kp\mc{S}\lp\mc{T}\in {\C}^{I_{1}\times
\cdots\times I_{N}\times J_{1}\times\cdots \times J_{M}}$, where
$\mc{R}\in {\C}^{I_{1}\times \cdots\times I_{N}\times
H_{1}\times\cdots \times H_{K}}$, $\mc{S}\in {\C}^{H_{1}\times
\cdots\times H_{K}\times G_{1}\times\cdots \times G_{L}}$ and
$\mc{T}\in {\C}^{G_{1}\times\cdots \times G_{L}\times
J_{1}\times\cdots \times J_{M}}$. Then,
$\mc{B}=\mc{A}^{\dg}(\mc{B}=\mc{A}_{\pi\dg})$ if and only if
$\mc{C}=\mc{A}_{\pi\dg}(\mc{C}=\mc{A}^{\dg})$.
\end{theorem}
\begin{proof}
Suppose that $\mc{B}=\mc{A}^{\dg}$. By equation \eqref{Eepgi}, we
have  $\mc{A}_{\pi\dg}=\mc{A}^{\dg}\n\mc{A}\m\mc{C}$. But
$\mc{C}=(\mc{R}^{\dg}\n\mc{A})^{\dg}\kp\mc{R}^{\dg}$ results
$\mc{A}_{\pi\dg}=\mc{C}$.

Conversely, suppose that $\mc{C}=\mc{A}_{\pi\dg}$. By equation
\eqref{ev} we get
$\mc{A}_{\pi\dg}=\mc{A}^{\dg}\n\mc{A}\m\mc{A}_{\pi\dg}$.
Pre-multiplying $\mc{A}^{\dg}\n\mc{A}$ to equation \eqref{eu}
results
$\mc{A}^{\dg}\n\mc{A}\m\mc{B}=\mc{A}_{\pi\dg}\n\mc{A}\m\mc{A}^{\dg}=\mc{B}.$
By Remark \ref{CaepgiU}, we have
$\mc{A}\m\mc{B}=\mc{A}\m\mc{A}^{\dg}$. Thus $\mc{A}^{\dg}=\mc{B}$.
\end{proof}

The next result gives a necessary and sufficient condition for
$\mc{A}_{\pi\dg}=\mc{T}^{\dg}\lp\mc{S}^{\dg}\kp\mc{R}^{\dg}$.
\begin{theorem}
Let $\mc{A}=\mc{R}\kp\mc{S}\lp\mc{T}\in {\C}^{I_{1}\times
\cdots\times I_{N}\times J_{1}\times\cdots \times J_{M}}$, where
$\mc{R}\in {\C}^{I_{1}\times \cdots\times I_{N}\times
H_{1}\times\cdots \times H_{K}}$, $\mc{S}\in {\C}^{H_{1}\times
\cdots\times H_{K}\times G_{1}\times\cdots \times G_{L}}$ and
$\mc{T}\in {\C}^{G_{1}\times\cdots \times G_{L}\times
J_{1}\times\cdots \times J_{M}}$. Then,
$\mc{A}_{\pi\dg}=\mc{T}^{\dg}\lp\mc{S}^{\dg}\kp\mc{R}^{\dg}$ if and
only if
$\mc{S}^{\dg}=(\mc{R}^{\dg}\n\mc{A}\m\mc{T}^{\dg})^{\dg}+\mc{Y}$,
where $\mc{Y}\in {\C}^{G_{1}\times\cdots \times G_{L}\times
H_{1}\times\cdots \times H_{K}} $ satisfies
$\mc{T}^{\dg}\lp\mc{Y}\kp\mc{R}^{\dg}=\mc{O}$.
\end{theorem}
\begin{proof}
Suppose that
$\mc{S}^{\dg}=(\mc{R}^{\dg}\n\mc{A}\m\mc{T}^{\dg})^{\dg}+\mc{Y}$. By
using equation \eqref{epgi}, we then have
$\mc{T}^{\dg}\lp\mc{S}^{\dg}\kp\mc{R}^{\dg}=\mc{T}^{\dg}\lp(\mc{R}^{\dg}\n\mc{A}\m\mc{T}^{\dg})^{\dg}\kp\mc{R}^{\dg}+\mc{T}^{\dg}\lp\mc{Y}\kp\mc{R}^{\dg}=\mc{A}_{\pi\dg}$.

Conversely, suppose
that$\mc{A}_{\pi\dg}=\mc{T}^{\dg}\lp\mc{S}^{\dg}\kp\mc{R}^{\dg}$,
then by equation \eqref{epgi} and the last equality of equation
\eqref{Enscpgi}
\begin{equation}
    \mc{T}\m\mc{T}^{\dg}\lp\mc{S}^{\dg}\kp\mc{R}^{\dg}\n\mc{R}=(\mc{R}^{\dg}\n\mc{A}\m\mc{T}^{\dg})^{\dg}.
\end{equation}
Hereafter, by applying equation \eqref{ir2} to
$\mc{T}\m\mc{T}^{\dg}\lp\mc{X}\kp\mc{R}^{\dg}\n\mc{R}=(\mc{R}^{\dg}\n\mc{A}\m\mc{T}^{\dg})^{\dg}$,
it follows that there exists a $\mc{Z}$ such that
\begin{equation}
    \mc{S}^{\dg}=(\mc{R}^{\dg}\n\mc{A}\m\mc{T}^{\dg})^{\dg}+\mc{Z}-\mc{T}\m\mc{T}^{\dg}\lp\mc{Z}\kp\mc{R}^{\dg}\n\mc{R},
\end{equation}
due to the equation \eqref{Enscpgi} and the fact
$(\mc{T}\m\mc{T}^{\dg})^{\dg}=\mc{T}\m\mc{T}^{\dg}$ and
$(\mc{R}^{\dg}\n\mc{R})^{\dg}=\mc{R}^{\dg}\n\mc{R}$. Thus, if
$\mc{Y}=\mc{Z}-\mc{T}\m\mc{T}^{\dg}\lp\mc{Z}\kp\mc{R}^{\dg}\n\mc{R}$,
then
\begin{eqnarray*}
    \mc{T}^{\dg}\lp\mc{Y}\kp\mc{R}^{\dg}&=&\mc{T}^{\dg}\lp\mc{Z}\kp\mc{R}^{\dg}-\mc{T}^{\dg}\lp\mc{T}\m\mc{T}^{\dg}\lp\mc{Y}\kp\mc{R}^{\dg}\n\mc{R}\kp\mc{R}^{\dg}\\
    &=&\mc{O}.
\end{eqnarray*}
\end{proof}

Next, we provide an example to show that it may be possible that
$\mc{A}^{\dg}=\mc{T}^{\dg}\lp\mc{S}^{\dg}\kp\mc{R}^{\dg}$ while
$\mc{A}_{\pi\dg}\neq\mc{A}^{\dg}$.

\begin{example}\label{exmppgi}
Let $\mc{A}=(a_{ijkl})\in{\C}^{2\times2\times2\times 2},~1\leq
i,j,k,l\leq2$. Suppose that $\mc{A}$ is factorized as
$\mc{A}=\mc{R}\2\mc{S}\2 \mc{T}$, where
$\mc{R}=(r_{ijkl})\in{\C}^{2\times2\times2\times
2},~\mc{S}=(s_{ijkl})\in{\C}^{2\times2\times2\times
2},~\mc{T}=(t_{ijkl})\in{\C}^{2\times2\times2\times 2},~1\leq
i,j,k,l\leq2$ such that

\vspace{0.5cm} \noindent
$$
a_{ij11}=
        \begin{pmatrix}
              1 & 1\\
              1 & 1
        \end{pmatrix},~
a_{ij12}=
        \begin{pmatrix}
              0 & 0\\
              0 & 0
        \end{pmatrix},~
a_{ij21}=
        \begin{pmatrix}
             0 & 1\\
             1 & 1
        \end{pmatrix},~
a_{ij22}=
        \begin{pmatrix}
            0 & 0\\
            0 & 0
        \end{pmatrix};
$$

\vspace{0.5cm} \noindent

$$
r_{ij11}=
         \begin{pmatrix}
                  1 & 0\\
                  0 & 0
         \end{pmatrix},~
r_{ij12}=
         \begin{pmatrix}
                      0 & 1\\
                      1 & 0
          \end{pmatrix},~
r_{ij21}=
          \begin{pmatrix}
                       0 & 0\\
                       1 & 0
          \end{pmatrix},~
r_{ij22}=
        \begin{pmatrix}
                     0 & 0\\
                     0 & 0
        \end{pmatrix};
$$

\vspace{0.5cm} \noindent
$$
s_{ij11}=
        \begin{pmatrix}
                     1 & 0\\
                     0 & 0
        \end{pmatrix},~
s_{ij12}=
        \begin{pmatrix}
                     0 & 1\slash2\\
                     -3\slash2 & 0
        \end{pmatrix},~
s_{ij21}=
        \begin{pmatrix}
                     0 & 1\slash2\\
                     1\slash2 & 0
        \end{pmatrix},~
s_{ij22}=
        \begin{pmatrix}
                     0 & 0\\
                     0 & 0
        \end{pmatrix};
$$

\vspace{0.5cm} \noindent and
$$
t_{ij11}=
        \begin{pmatrix}
                     1 & 1\\
                     1 & 1
        \end{pmatrix},~
t_{ij12}=
        \begin{pmatrix}
                     0 & 1\\
                     1 & 1
        \end{pmatrix},~
t_{ij21}=
        \begin{pmatrix}
                     0 & 1\\
                     1 & 1
        \end{pmatrix},~
t_{ij22}=
        \begin{pmatrix}
                     0 & 0\\
                     0 & 0
        \end{pmatrix}.
$$

\noindent Now, $\mc{A}^{\dg}=(a'_{ijkl})$,
$\mc{R}^{\dg}=(r'_{ijkl})$, $\mc{S}^{\dg}=(s'_{ijkl})$, and
$\mc{T}^{\dg}=(t'_{ijkl})$ are given by
$$
a'_{ij11}=
         \begin{pmatrix}
                      1 & 0\\
                      -1 & 0
        \end{pmatrix},~
a'_{ij12}=
         \begin{pmatrix}
                      0 & 0\\
                      1\slash3 & 0
        \end{pmatrix},~
a'_{ij21}=
         \begin{pmatrix}
                      0 & 0\\
                      1\slash3 & 0
        \end{pmatrix},~
a'_{ij22}=
         \begin{pmatrix}
                      0 & 0\\
                      1\slash3 & 0
        \end{pmatrix};
$$

\vspace{0.5cm} \noindent
$$
r'_{ij11}=
         \begin{pmatrix}
                      1 & 0\\
                      0 & 0
         \end{pmatrix},~
r'_{ij12}=
         \begin{pmatrix}
                      0 & 1\\
                      -1 & 0
        \end{pmatrix},~
r'_{ij21}=
         \begin{pmatrix}
                      0 & 0\\
                      1 & 0
         \end{pmatrix},~
r'_{ij22}=
         \begin{pmatrix}
                      0 & 0\\
                      0 & 0
         \end{pmatrix};
$$

\vspace{0.5cm} \noindent
$$
s'_{ij11}=
         \begin{pmatrix}
                      1 & 0\\
                      0 & 0
         \end{pmatrix},~
s'_{ij12}=
         \begin{pmatrix}
                      0 & 1\slash2\\
                      3\slash2 & 0
        \end{pmatrix},~
s'_{ij21}=
         \begin{pmatrix}
                      0 & -1\slash2\\
                      1\slash2 & 0
         \end{pmatrix},~
s'_{ij22}=
         \begin{pmatrix}
                      0 & 0\\
                      0 & 0
         \end{pmatrix};
$$

\vspace{0.5cm} \noindent and
$$
t'_{ij11}=
         \begin{pmatrix}
                      1 & -1\slash2\\
                      -1\slash2 & 0
         \end{pmatrix},~
t'_{ij12}=
         \begin{pmatrix}
                      0 & 1\slash6\\
                      1\slash6 & 0
         \end{pmatrix},~
t'_{ij21}=
         \begin{pmatrix}
                      0 & 1\slash6\\
                      1\slash6 & 0
         \end{pmatrix},~
t'_{ij22}=
         \begin{pmatrix}
                     0 & 1\slash6\\
                      1\slash6 & 0
         \end{pmatrix}.
$$

\vspace{0.2cm} \noindent Here,
$\mc{T}^{\dg}\2\mc{S}^{\dg}\2\mc{R}^{\dg}=(y_{ijkl})$, where
 \begin{eqnarray*}
y_{ij11}=
         \begin{pmatrix}
                      1 & 0\\
                      -1 & 0
        \end{pmatrix},~
y_{ij12}=
         \begin{pmatrix}
                      0 & 0\\
                      1\slash3 & 0
        \end{pmatrix},~
y_{ij21}=
         \begin{pmatrix}
                      0 & 0\\
                      1\slash3 & 0
        \end{pmatrix},~
y_{ij22}=
         \begin{pmatrix}
                      0 & 0\\
                      1\slash3 & 0
        \end{pmatrix}.
\end{eqnarray*}
But,
$\mc{A}_{\pi\dg}=\mc{T}^{\dg}\2(\mc{R}^{\dg}\2\mc{A}\2\mc{T}^{\dg})^{\dg}\2\mc{R}^{\dg}=(x_{ijkl})$,
where
 \begin{eqnarray*}
 x_{ij11}=
         \begin{pmatrix}
                      1 & 1\slash2\\
                      -1 & 0
         \end{pmatrix},~
x_{ij12}=
         \begin{pmatrix}
                      0 & -1\slash6\\
                      1\slash3 & 0
         \end{pmatrix},~
x_{ij21}=
        \begin{pmatrix}
                     0 & -1\slash6\\
                      1\slash3 & 0
        \end{pmatrix},~
x_{ij22}=
        \begin{pmatrix}
                     0 & -1\slash6\\
                      1\slash3 & 0
        \end{pmatrix}.
\end{eqnarray*}
Here, in this example $\mc{A}_{\pi\dg}\neq\mc{A}^{\dg}$, while
$\mc{A}^{\dg}=\mc{T}^{\dg}\2\mc{S}^{\dg}\2\mc{R}^{\dg}$.
\end{example}

In case of the Moore-Penrose inverse, we observed that the
Moore-Penrose inverse of $\mc{A}^{\dg}$ reduces to $\mc{A}$ for
every tensor $\mc{A}$ \cite{sun}. This fact leads to a natural
question ``Does the product Moore-Penrose inverse of
$\mc{A}_{\pi\dg}$ also reduces to $\mc{A}$?'' The answer is
\textit{affirmative}, if we use the same factorization employed in
forming $\mc{A}_{\pi\dg}$. The following theorem proves this fact.

\begin{theorem}
Let $\mc{A}=\mc{R}\kp\mc{S}\lp\mc{T}\in {\C}^{I_{1}\times
\cdots\times I_{N}\times J_{1}\times\cdots \times J_{M}}$, where
$\mc{R}\in {\C}^{I_{1}\times \cdots\times I_{N}\times
H_{1}\times\cdots \times H_{K}}$, $\mc{S}\in {\C}^{H_{1}\times
\cdots\times H_{K}\times G_{1}\times\cdots \times G_{L}}$ and
$\mc{T}\in {\C}^{G_{1}\times\cdots \times G_{L}\times
J_{1}\times\cdots \times J_{M}}$. Then,
$$(\mc{A}_{\pi\dg})_{\pi\dg}=[\mc{T}^{\dg}\lp(\mc{R}^{\dg}\n\mc{A}\m\mc{T}^{\dg})^{\dg}\kp\mc{R}^{\dg}]_{\pi\dg}=\mc{A}.$$
\end{theorem}
\begin{proof}
Applying equation \eqref{epgi} to
$\mc{A}_{\pi\dg}=\mc{T}^{\dg}\lp(\mc{R}^{\dg}\n\mc{A}\m\mc{T}^{\dg})^{\dg}\kp\mc{R}^{\dg}$
gives
\begin{equation}\label{le}
    (\mc{A}_{\pi\dg})_{\pi\dg}=\mc{R}\kp[\mc{T}\m\mc{T}^{\dg} \lp(\mc{R}^{\dg}\n\mc{A}\m\mc{T}^{\dg})^{\dg}\kp\mc{R}^{\dg}\n\mc{R}]^{\dg}\lp\mc{T}.
\end{equation}
With the help of Lemma \ref{ir}, equation \eqref{le} gives
$(\mc{A}_{\pi\dg})_{\pi\dg}=\mc{A}$.

\end{proof}

\section{Conclusion}

Note that in Example \ref{exmppgi},
$\mc{A}^{\dg}=\mc{T}^{\dg}\lp\mc{S}^{\dg}\kp\mc{R}^{\dg}$ holds.  A
natural question arises now  regarding the trueness of the same
equality which is called as the \textit{triple reverse order law} of
tensors via the Einstein product, i.e., $(\mc{R} \kp \mc{S} \lp
\mc{T})^{\dg}=\mc{T}^{\dg}\lp\mc{S}^{\dg}\kp\mc{R}^{\dg}$.  However,
this is not true in general, and  is shown next through an example.

\begin{example}
Let $\mc{R}=(r_{ijkl})\in{\C}^{2\times2\times2\times
2},~\mc{S}=(s_{ijkl})\in{\C}^{2\times2\times2\times
2},~\mc{T}=(t_{ijkl})\in{\C}^{2\times2\times2\times 2},~1\leq
i,j,k,l\leq2$ such that

\vspace{0.5cm}

\noindent
$$
r_{ij11}=
         \begin{pmatrix}
                  1 & 0\\
                  -1 & 0
         \end{pmatrix},~
r_{ij12}=
         \begin{pmatrix}
                      1 & 1\\
                      0 & 1
          \end{pmatrix},~
r_{ij21}=
          \begin{pmatrix}
                       0 & 1\\
                       0 & 1
          \end{pmatrix},~
r_{ij22}=
        \begin{pmatrix}
                     -1 & 0\\
                     1 & 0
        \end{pmatrix};
$$

\vspace{0.5cm} \noindent
$$
s_{ij11}=
        \begin{pmatrix}
                     0 & 1\\
                     0 & -1
        \end{pmatrix},~
s_{ij12}=
        \begin{pmatrix}
                     1 & 1\\
                     1 & 0
        \end{pmatrix},~
s_{ij21}=
        \begin{pmatrix}
                     1 & 0\\
                     1 & 0
        \end{pmatrix},~
s_{ij22}=
        \begin{pmatrix}
                     0 & -1\\
                     0 & 1
        \end{pmatrix};
$$

\vspace{0.5cm} \noindent and
$$
t_{ij11}=
        \begin{pmatrix}
                     1 & 0\\
                     0 & 0
        \end{pmatrix},~
t_{ij12}=
        \begin{pmatrix}
                     0 & 0\\
                     0 & 1
        \end{pmatrix},~
t_{ij21}=
        \begin{pmatrix}
                     0 & 0\\
                     0 & 0
        \end{pmatrix},~
t_{ij22}=
        \begin{pmatrix}
                     0 & 1\\
                     0 & 0
        \end{pmatrix}.
$$

\vspace{0.2cm} \noindent
And let $\mc{A}=\mc{R}\2\mc{S}\2 \mc{T}=(a_{ijkl})\in{\C}^{2\times2\times2\times 2},~1\leq i,j,k,l\leq2$, where \\

\vspace{0.2cm} \noindent
$$
a_{ij11}=
        \begin{pmatrix}
              -1 & -1\\
              0 & 1
        \end{pmatrix},~
a_{ij12}=
        \begin{pmatrix}
              1 & 0\\
              0 & 1
        \end{pmatrix},~
a_{ij21}=
        \begin{pmatrix}
             1 & 1\\
             0 & 0
        \end{pmatrix},~
a_{ij22}=
        \begin{pmatrix}
            1 & 1\\
            0 & -1
        \end{pmatrix}.
$$

\vspace{0.2cm} \noindent
Now, $\mc{A}^{\dg}=(a'_{ijkl})$, $\mc{R}^{\dg}=(r'_{ijkl})$, $\mc{S}^{\dg}=(s'_{ijkl})$, and $\mc{T}^{\dg}=(t'_{ijkl})$ are given by \\

\vspace{0.2cm} \noindent
$$
r'_{ij11}=
         \begin{pmatrix}
                      0 & 1\\
                      -1 & 0
         \end{pmatrix},~
r'_{ij12}=
         \begin{pmatrix}
                      0 & 0\\
                      1\slash 2 & 0
        \end{pmatrix},~
r'_{ij21}=
         \begin{pmatrix}
                      -1\slash2 & 1\\
                      -1 & 1\slash 2
         \end{pmatrix},~
r'_{ij22}=
         \begin{pmatrix}
                      0 & 0\\
                      1\slash 2 & 0
         \end{pmatrix};
$$

\vspace{0.5cm} \noindent
$$
s'_{ij11}=
         \begin{pmatrix}
                      10 & 0\\
                      1\slash 2 & 0
         \end{pmatrix},~
s'_{ij12}=
         \begin{pmatrix}
                      0 & 1\\
                      -1 & 0
        \end{pmatrix},~
s'_{ij21}=
         \begin{pmatrix}
                      0 & 0\\
                      1\slash2 & 0
         \end{pmatrix},~
s'_{ij22}=
         \begin{pmatrix}
                      -1\slash 2 & 1\\
                      -1 & 1\slash 2
         \end{pmatrix};
$$

\vspace{0.5cm} \noindent and
$$
t'_{ij11}=
         \begin{pmatrix}
                      1 & 0\\
                      0 & 0
         \end{pmatrix},~
t'_{ij12}=
         \begin{pmatrix}
                      0 & 0\\
                      0 & 1
         \end{pmatrix},~
t'_{ij21}=
         \begin{pmatrix}
                      0 & 0\\
                      0 & 0
         \end{pmatrix},~
t'_{ij22}=
         \begin{pmatrix}
                     0 & 1\\
                     0 & 0
         \end{pmatrix}.
$$

\vspace{0.2cm} \noindent
Here, $\mc{A}^{\dg}=(\mc{R}\2\mc{S}\2\mc{T})^{\dg}=(a'_{ijkl})$, where \\

\vspace{0.2cm} \noindent
$$
a'_{ij11}=
         \begin{pmatrix}
                      -1\slash 2 & 1\\
                      -1 & 1\slash 2
        \end{pmatrix},~
a'_{ij12}=
         \begin{pmatrix}
                      1\slash 2 & -1\\
                      2 & -1\slash 2
        \end{pmatrix},~
a'_{ij21}=
         \begin{pmatrix}
                      0 & 0\\
                      0 & 0
        \end{pmatrix},~
a'_{ij22}=
         \begin{pmatrix}
                      1\slash 2 & 0\\
                      1 & -1\slash 2
        \end{pmatrix}.
$$

\vspace{0.2cm} \noindent But,
$\mc{T}^{\dg}\2\mc{S}^{\dg}\2\mc{R}^{\dg}=(x_{ijkl})$, where
 \begin{eqnarray*}
x_{ij11}=
         \begin{pmatrix}
                      -1\slash 4 & 1\slash2\\
                      -1\slash 2 & 1\slash 4
         \end{pmatrix},~
x_{ij12}=
         \begin{pmatrix}
                      1\slash 2 & -3\slash2\\
                      9\slash4 & -1\slash2
         \end{pmatrix},~
x_{ij21}=
        \begin{pmatrix}
                     0 & 0\\
                     0 & 0
        \end{pmatrix},~
x_{ij22}=
        \begin{pmatrix}
                     1\slash2 & -1\\
                      3\slash2 & -1\slash 2
        \end{pmatrix}.
\end{eqnarray*}
Here, in this example
$(\mc{R}\2\mc{S}\2\mc{T})^{\dg}\neq\mc{T}^{\dg}\2\mc{S}^{\dg}\2\mc{R}^{\dg}$.
\end{example}

The following theorem gives a necessary and sufficient condition for
the triple reverse order law.

\begin{theorem}\label{trol1}
Let $\mc{A}=\mc{R}\kp\mc{S}\lp\mc{T}\in {\C}^{I_{1}\times
\cdots\times I_{N}\times J_{1}\times\cdots \times J_{M}}$, where
$\mc{R}\in {\C}^{I_{1}\times \cdots\times I_{N}\times
H_{1}\times\cdots \times H_{K}}$, $\mc{S}\in {\C}^{H_{1}\times
\cdots\times H_{K}\times G_{1}\times\cdots \times G_{L}}$ and
$\mc{T}\in {\C}^{G_{1}\times\cdots \times G_{L}\times
J_{1}\times\cdots \times J_{M}}$. Then,
$\mc{A}^{\dg}=\mc{T}^{\dg}\lp\mc{S}^{\dg}\kp\mc{R}^{\dg}$ if and
only if
\begin{itemize}
    \item[(i)] $\mc{R}^{\dg}\n(\mc{R}\kp\mc{S}\lp\mc{T})\m(\mc{T}^{\dg}\lp\mc{S}^{\dg}\kp\mc{R}^{\dg})\n(\mc{R}\kp\mc{S}\lp\mc{T})\m\mc{T}^{\dg}=\mc{R}^{\dg}\n(\mc{R}\kp\mc{S}\lp\mc{T})\m\mc{T}^{\dg},$
    \item[(ii)] $\mc{T}\m (\mc{T}^{\dg}\lp\mc{S}^{\dg}\kp\mc{R}^{\dg})\n(\mc{R}\kp\mc{S}\lp\mc{T})\m(\mc{T}^{\dg}\lp\mc{S}^{\dg}\kp\mc{R}^{\dg})\n\mc{R}=\mc{T}\m(\mc{T}^{\dg}\lp\mc{S}^{\dg}\kp\mc{R}^{\dg})\n\mc{R},$
    \item[(iii)]$[\mc{R}^{H}\n(\mc{R}\kp\mc{S}\lp\mc{T})\m(\mc{T}^{\dg}\lp\mc{S}^{\dg}\kp\mc{R}^{\dg})\n\mc{R}]^{H}=\mc{R}^{H}\n(\mc{R}\kp\mc{S}\lp\mc{T})\m(\mc{T}^{\dg}\lp\mc{S}^{\dg}\kp\mc{R}^{\dg})\n\mc{R},$
    \item[(iv)]$[\mc{T} \m(\mc{T}^{\dg}\lp\mc{S}^{\dg}\kp\mc{R}^{\dg})\n(\mc{R}\kp\mc{S}\lp\mc{T})\m\mc{T}^{H}]^{H}=\mc{T}\m (\mc{T}^{\dg}\lp\mc{S}^{\dg}\kp\mc{R}^{\dg})\n(\mc{R}\kp\mc{S}\lp\mc{T})\m\mc{T}^{H}$.
\end{itemize}

\end{theorem}
\begin{proof}
It can be esily seen that $(i)$ and $(ii)$ are equivalent to
$(\mc{R}\kp\mc{S}\lp\mc{T})\m(\mc{T}^{\dg}\lp\mc{S}^{\dg}\kp\mc{R}^{\dg})\\\n(\mc{R}\kp\mc{S}\lp\mc{T})=\mc{R}\kp\mc{S}\lp\mc{T}$
and $
(\mc{T}^{\dg}\lp\mc{S}^{\dg}\kp\mc{R}^{\dg})\n(\mc{R}\kp\mc{S}\lp\mc{T})\m(\mc{T}^{\dg}\lp\mc{S}^{\dg}\kp\mc{R}^{\dg})=\mc{T}^{\dg}\lp\\\mc{S}^{\dg}\kp\mc{R}^{\dg}$,
respectively. $(iii)$ is equivalent to
$[(\mc{R}\kp\mc{S}\lp\mc{T})\m(\mc{T}^{\dg}\lp\mc{S}^{\dg}\kp\mc{R}^{\dg})]^{H}=(\mc{R}\kp\mc{S}\lp\\\mc{T})\m(\mc{T}^{\dg}\lp\mc{S}^{\dg}\kp\mc{R}^{\dg})$.
Because
$[(\mc{R}\kp\mc{S}\lp\mc{T})\m(\mc{T}^{\dg}\lp\mc{S}^{\dg}\kp\mc{R}^{\dg})]^{H}=(\mc{R}\kp\mc{S}\lp\mc{T})\m(\mc{T}^{\dg}\lp\mc{S}^{\dg}\\\kp\mc{R}^{\dg})$
imply $(iii)$ and conversely, by pre-multiplying  $\mc{R}^{\dg *}$
and post-multiplying $\mc{R}^{\dg}$ to $(iii)$ gives
$[(\mc{R}\kp\mc{S}\lp\mc{T})\m(\mc{T}^{\dg}\lp\mc{S}^{\dg}\kp\mc{R}^{\dg})]^{H}=(\mc{R}\kp\mc{S}\lp\mc{T})\m(\mc{T}^{\dg}\lp\mc{S}^{\dg}\kp\mc{R}^{\dg})$.
Similarly, $(iv)$ is equivalent to
$[(\mc{T}^{\dg}\lp\mc{S}^{\dg}\kp\mc{R}^{\dg})\n(\mc{R}\kp\mc{S}\lp\mc{T})]^{H}=
(\mc{T}^{\dg}\lp\mc{S}^{\dg}\kp\mc{R}^{\dg})\n(\mc{R}\kp\mc{S}\lp\mc{T})$.
\end{proof}

 Note that the triple reverse order law  presented in the above theorem contains a number of conditions. We conclude the article with the note that  simpler characterization of triple reverse order law will be tried in near future.


{\small {\bf Acknowledgments.}\\

The authors acknowledge the support provided by Science and
Engineering Research Board, Department of Science and Technology,
New Delhi, India, under the grant number YSS/2015/000303. }

\section*{References}

\bibliographystyle{amsplain}

\end{document}